\documentclass[11pt,twoside]{article}
\usepackage{amsmath,amssymb,amsthm}
\usepackage{amsmath,amssymb,amsthm}
\usepackage{graphicx}
\usepackage{epstopdf}
\usepackage{caption}
\usepackage{cite}
\allowdisplaybreaks[4]

\newtheorem{lm}{Lemma}[section]
\newtheorem{thm}{Theorem}[section]

\newcounter{saveeqn}%

\makeatletter
\evensidemargin
\oddsidemargin
\makeatother

\makeatletter
\@addtoreset{equation}{section}
\makeatother

\title{\Large\bf Crossing limit cycles of nonsmooth Li\'enard systems
and applications
\thanks{
Supported by NSFC 11871355 and 11801079, Graduate Student's Research and Innovation Fund of Sichuan University 2018YJSY047.
}
}

\author{Tao Li$^1$, ~~~~Hebai Chen$^2$, ~~~~Xingwu Chen$^{1}$\!\!
\footnote{Author to whom any correspondence should be addressed. Email address: xingwu.chen@hotmail.com (X. Chen).}
\\
{\small 1. Department of Mathematics, Sichuan University,}\\
{\small Chengdu, Sichuan 610064, P. R. China}
\\
{\small 2. School of Mathematics and Statistics, Central South University,}\\
  {\small Changsha, Hunan 410083, P. R. China}
}
\date{}

\begin{document}
\maketitle

\begin{abstract}
Continuing the investigation for the number of crossing limit cycles of
nonsmooth Li\'enard systems in [Nonlinearity {\bf 21}(2008), 2121-2142] for
the case of a unique equilibrium, in this paper we consider the case of any number of equilibria.
We give results about the existence and uniqueness of crossing limit cycles, which hold
not only for a unique equilibrium but also for multiple equilibria.
Moreover, we find a sufficient condition for the nonexistence of crossing limit cycles.
Finally, applying our results we prove the uniqueness of crossing limit cycles for planar piecewise linear systems
with a line of discontinuity and without sliding sets.
\vskip 0.2cm
{\bf Keywords:} discontinuity, Li\'enard systems, limit cycles, piecewise linear systems
\end{abstract}

\baselineskip 15pt
\parskip 10pt
\thispagestyle{empty}
\setcounter{page}{1}

\section{Introduction and main results}
\setcounter{equation}{0}
\setcounter{lm}{0}
\setcounter{thm}{0}
\setcounter{rmk}{0}
\setcounter{df}{0}
\setcounter{cor}{0}
A class of important dynamical systems is the Li\'enard system, which is originated from physics and then is applied to
engineering, biology, chemistry and more fields. As usual, the Li\'enard system can be written as
\begin{eqnarray}
\left\{
\begin{aligned}
&\dot x=F(x)-y,\\
&\dot y=g(x),
\end{aligned}
\right.
\label{ls}
\end{eqnarray}
where $F(x):=\int_0^xf(s)ds$, $f(x)$ and $g(x)$ are two scalar functions.
For system (\ref{ls}), a main and challenging subject is to study the existence, uniqueness and number of limit cycles,
i.e., isolated periodic orbits in the phase space.
The 13th Smale's problem provided in \cite{SM} is about the maximum number of limit cycles in the polynomial system of form $\dot x=F(x)-y, \dot y=x$,
i.e., the famous 16th Hilbert's problem restricted to the polynomial system of form (\ref{ls}) with $g(x)=x$, and is still open.
When system (\ref{ls}) is smooth, the investigation of the existence, uniqueness and number of limit cycles has a long history and
many excellent results are obtained as given in journal papers \cite{FL, FL2, FP, CHRI, CLJL} and text book \cite{ZZF}.

On the other hand, many models in practical problems are established by system (\ref{ls}) in a nonsmooth form such as
\begin{eqnarray}
\left\{
\begin{aligned}
&\dot x=ax+bx^3-y,\\
&\dot y=x-{\rm sgn}(x),
\end{aligned}
\right.
\label{e1}
\end{eqnarray}
which is the limit case of a smooth oscillator (see \cite{ChL-QC, C1}). Another system
\begin{eqnarray}
\left\{
\begin{aligned}
&\dot x=Tx-y,\\
&\dot y=Dx-\left\{
\begin{aligned}
&1-D\cdot X_{ref}~~~~~~&&{\rm if}~~x<0,\\
&-D\cdot X_{ref}~~~~~~&&{\rm if}~~x>0
\end{aligned}
\right.
\end{aligned}
\right.
\label{e2}
\end{eqnarray}
is the Li\'enard form of a buck electronic converter (see \cite{ChL-EE}),
where $T, D, X_{ref}$ are parameters. Motivated by practical problems, many mathematicians and
engineers have started to investigate the existence, uniqueness and number of limit cycles for the nonsmooth
Li\'enard system (\ref{ls}). Usually, the nonsmoothness leads to much difficulty in the analysis of nonsmooth
Li\'enard systems and much less results are obtained(see,e.g., \cite{CHX, FJJ, JM, JEF}), compared with the smooth case.

Consider the Li\'enard system (\ref{ls}) with nonsmooth functions
$$f(x):=
\left\{
\begin{array}{c}
f^+(x)~~~~~~{\rm if}~~x>0,\\
f^-(x)~~~~~~{\rm if}~~x<0,
\end{array}
\right.
~~~~~~~~~~~~~~~~
g(x):=
\left\{
\begin{array}{c}
g^+(x)~~~~~~{\rm if}~~x>0,\\
g^-(x)~~~~~~{\rm if}~~x<0,
\end{array}
\right.$$
where $f^\pm(x), g^\pm(x):\mathbb{R}\rightarrow\mathbb{R}$ are smooth functions.
Hence, the nonsmoothness or even discontinuity of system (\ref{ls}) occurs only on $y$-axis, the {\it switching line}.
Let $F^\pm(x):=\int_{0}^xf^\pm(s)ds$. We rewrite the nonsmooth Li\'enard system (\ref{ls}) as
\begin{eqnarray}
\left(\dot x, \dot y\right)=
\left\{
\begin{aligned}
\left(F^+(x)-y, g^+(x)\right)~~~~~~~~{\rm if}~~x>0,\\
\left(F^-(x)-y, g^-(x)\right)~~~~~~~~{\rm if}~~x<0.
\end{aligned}
\right.
\label{LS}
\end{eqnarray}
For convenience, we call the subsystem in $x>0$ and $x<0$ the {\it right system} and {\it left system} of (\ref{LS}), respectively.
Since $F^\pm(0)=0$, we always define the $x$-component of the vector field of (\ref{LS}) as $-y$ on $y$-axis, i.e., $\dot x=-y$.
If $g^+(0)=g^-(0)$, the $y$-component of the vector field of (\ref{LS}) is defined as $g^+(0)$ on $y$-axis, i.e., $\dot y=g^+(0)$.
Then (\ref{LS}) is continuous.
If $g^+(0)\ne g^-(0)$, the $y$-component of the vector field of (\ref{LS}) has a
jump discontinuity on $y$-axis, which implies that (\ref{LS}) is a discontinuous system. Thus we define
the solution of (\ref{LS}) passing through a point in $y$-axis by the Filippov convention(see \cite{A, YSA}).
In fact, observe that the switching line consists of the origin $O$ and two {\it crossing sets}, i.e., the positive $y$-axis
and the negative $y$-axis. Let $q$ be a point in $y$-axis. If $q$ belongs to the positive (resp. negative) $y$-axis, then
the solution of (\ref{LS}) passing through $q$ crosses $y$-axis at $q$ from right (resp. left) to left (resp. right).
If $q$ lies at $O$, it is proved in \cite[Proposition 1]{JEF} that $q$ is a {\it boundary equilibrium} when
$g^+(0)g^-(0)=0$, a {\it pseudo-equilibrium} when $g^+(0)g^-(0)<0$ and a regular point when $g^+(0)g^-(0)>0$. Thus an equilibrium of (\ref{LS})
has three types, including boundary equilibria and pseudo-equilibria lying in $x=0$, regular equilibria lying in $x\ne0$.
In the third case, $q$ is also called a {\it parabolic fold-fold point} (see \cite{BMT}) and, more precisely, both orbits of the left and right systems
passing through $q$ are quadratically tangent to the $y$-axis from the left half plane when $g^+(0)>0, g^-(0)>0$
and from the right half plane when $g^+(0)<0, g^-(0)<0$.

A limit cycle of system (\ref{LS}) either totally lies in one of half planes $x\ge0$ and $x\le0$, or presents intersections with both $x>0$ and $x<0$.
Since the former one can be determined by one of the right and left systems, our attention is paid to the latter one, i.e., {\it crossing limit cycle}
(see \cite{YSA}). For (\ref{LS}), there exist a few results on the existence, uniqueness and number of crossing
limit cycles, such as \cite{CHX, FJJ, JM, JEF}. In \cite{JEF}, a necessary condition of the existence of crossing limit cycles
and a sufficient condition for the uniqueness are given. In \cite{FJJ}, a sufficient condition of the existence and uniqueness of
crossing limit cycles is presented. The number of crossing limit cycles is studied in \cite{CHX,JM}. Unfortunately,
we observe that almost in all publications (\ref{LS}) is required to satisfy
\begin{eqnarray}
g^+(x)>0~~~~ {\rm for}~~ x>0, ~~~~~~~g^-(x)<0~~~~ {\rm for}~~ x<0.
\label{dapo}
\end{eqnarray}
Condition (\ref{dapo}) implies that the origin $O$ is the unique equilibrium of (\ref{LS})
and, hence, any crossing limit cycle surrounds $O$ as indicated in \cite{JEF}.
However, in many practical problems system (\ref{LS}) may have multiple equilibria that lie in the left and right half planes,
such as systems (\ref{e1}) and (\ref{e2}), and then a crossing limit cycle can surround multiple equilibria.
On the existence, uniqueness and number of crossing limit cycles of (\ref{LS}) with multiple equilibria,
as far as we know, the results are lacking and there are only a few results restricted to concrete models (see e.g.,\cite{C1}). For smooth Li\'enard systems with multiple equilibria, some results of limit cycles have been given in \cite{FL, FL2, XDZZ, ZZF}.

The goal of this paper is to study the existence, nonexistence and uniqueness of crossing limit cycles for the nonsmooth Li\'enard system (\ref{LS})
with any number of equilibria. We particularly emphasize the case of multiple equilibria. In order to state our main results, we give some basic hypotheses for system (\ref{LS}) as the following {\bf(H1)}-{\bf(H3)}.
\begin{description}
\item[(H1)] There exists a constant $x_e\ge0$ such that $g^+(x)(x-x_e)>0$ for $x>0$ and $x\ne x_e$.
\item[(H2)] $f^+(x)>0$ for $x>0$ and $f^-(x)<0$ for $x<0$.
\end{description}

Define
\begin{eqnarray}
p=p(x):=\left\{
\begin{aligned}
F^+(x)~~~~~{\rm if}~~x\ge0,\\
F^-(x)~~~~~{\rm if}~~x<0.
\end{aligned}
\right.
\label{pde}
\end{eqnarray}
Clearly, $p(x)$ is continuous and $p(0)=0$ due to $F^\pm(0)=0$. Moreover, $p'(x)=f^+(x)>0$ for $x>0$ and $p'(x)=f^-(x)<0$ for $x<0$
by {\bf(H2)}. Thus $p(x)$ is strictly increasing for $x>0$ and strictly decreasing for $x<0$, implying $p(x)\ge0$ for
all $x\in \mathbb{R}$ and $p(x)$ has a strictly increasing inverse function
\begin{eqnarray}
x^+(p):[0, p^+)\rightarrow[0, +\infty)
\label{inverser}
\end{eqnarray}
and a strictly decreasing inverse function
\begin{eqnarray}
x^-(p):[0, p^-)\rightarrow(-\infty, 0],
\label{inversel}
\end{eqnarray}
where $p(x)\rightarrow p^\pm (x\rightarrow\pm\infty)$. It follows from the monotonicity that $p^+$ (resp. $p^-$) is either a constant
or infinity.
\begin{description}
\item[(H3)] There exist the limits
\begin{eqnarray}
\lim_{p\rightarrow0^+}\frac{g^\pm(x^\pm(p))}{f^\pm(x^\pm(p))}=\eta^\pm
\label{limit}
\end{eqnarray}
and $-\infty\!<\!\eta^+\!\le\!\eta^-\!<\!+\infty$.
Moreover, $g^+/f^+|_{x=x^+(p)}<g^-/f^-|_{x=x^-(p)}$ for all sufficiently small $p>0$ if $\eta^+=\eta^-$.
\end{description}

Under hypothesis {\bf(H1)}, in the right half plane system (\ref{LS}) has no equilibria if $x_e=0$ and a unique equilibrium if $x_e>0$.
Moreover, the unique equilibrium in the right half plane lies at $(x_e, F^+(x_e))$. Let $E$ be the point lying at $(x_e, F^+(x_e))$ if $x_e>0$ and $O$ if $x_e=0$. The number of equilibria in the left half plane is not determined by {\bf(H1)}-{\bf(H3)}. In other word, system (\ref{LS}) in the left half plane may have any number of equilibria. On the $y$-axis, system (\ref{LS})
has no equilibria if $g^+(0)g^-(0)>0$ and a unique equilibrium $O$ if $g^+(0)g^-(0)\le 0$ as indicated in
the third paragraph. The example satisfying $f^+(x)=1, f^-(x)=-1, g^+(x)=x-2a$ and $g^-(x)=2x+a$ with $a\ge0$
implies the reasonability of {\bf(H1)}-{\bf(H3)}.

In the following, we state our main results.

\begin{thm}
If system {\rm(\ref{LS})} with {\bf (H1)-(H3)} has a crossing periodic orbit $\Gamma$, then
\\
{\rm(i)} $\Gamma$ surrounds $O$ and $E$ counterclockwise;
\\
{\rm(ii)} the equations
\begin{eqnarray}
F^-(x^-)=F^+(x^+),~~~~~~\frac{g^-(x^-)}{f^-(x^-)}=\frac{g^+(x^+)}{f^+(x^+)}
\label{eq}
\end{eqnarray}
have at least one solution $(x^-, x^+)=(x^-_*, x^+_*)$ with $x^-_*<0<x^+_*$
satisfying that $\Gamma$ transversally intersects both lines $x=x^-_*$ and $x=x^+_*$.
\label{suff}
\end{thm}

Theorem~\ref{suff} provides two necessary conditions for the existence of crossing periodic orbits for system (\ref{LS}).
These two necessary conditions help us to determine the configuration of crossing periodic orbits, that is,
any crossing periodic orbit neither lie in each side of the line $x=x_e$ nor
lie in the strip $\hat x^-_*<0<\hat x^+_*$, where $(\hat x^-_*, \hat x^+_*)$ is the solution of (\ref{eq}) satisfying that $\hat x^-_*<0<\hat x^+_*$ is the narrowest strip.
On the other hand, Theorem~\ref{suff} can be regarded as a generalization of \cite[Theorem 2]{JEF}
from one equilibrium to multiple equilibria.
It is required in \cite{JEF} that system (\ref{LS}) satisfies the condition (\ref{dapo}), implying that $O$ is
the unique equilibrium of (\ref{LS}) and any crossing periodic orbit surrounds $O$.
However, in this paper we replace (\ref{dapo}) by the weaker hypothesis {\bf(H1)}, which allows (\ref{LS}) to have multiple equilibria and crossing limit cycles surrounding multiple equilibria.
For example, by Theorem~\ref{suff} the crossing periodic orbit $\Gamma$ surrounds at least
two equilibria ($O$ and $E$) when $x_e>0$ and $g^+(0)g^-(0)\le 0$.
Therefore, our result holds not only for a unique equilibrium, but also for multiple equilibria.

\begin{thm}
For system {\rm(\ref{LS})} with {\bf(H1)-(H3)}, assume that the equations in {\rm(\ref{eq})} have a unique solution
$(x^-, x^+)=(x^-_*, x^+_*)$ with $x^-_*<0<x^+_*$ and one of the following hypotheses holds:
\begin{description}
\setlength{\itemsep}{0mm}
\item[(H4)] $x^+_*\ge x_e$ and $F^+(x)f^+(x)/g^+(x)$ is increasing for $x>x_*^+$;
\item[(H5)] $K^-(x^-(p_2))<K^+(x^+(p_1))$ for all $p_1, p_2$ satisfying $p_2>p_1\ge F^+(x^+_*)$, where
$$K^\pm(x):=\frac{(g^\pm(x))'f^\pm(x)-(f^\pm(x))'g^\pm(x)}{(f^\pm(x))^3}$$
and $x^\pm(p)$ are given in {\rm(\ref{inverser})} and {\rm(\ref{inversel})} respectively.
\end{description}
Then system {\rm(\ref{LS})} has at most one crossing periodic orbit,
which is a stable and hyperbolic crossing limit cycle if it exists.
\label{uni}
\end{thm}

Theorem~\ref{uni} is a result about the uniqueness of crossing limit cycles for system (\ref{LS}).
In the aspect of the number of equilibria,
Theorem~\ref{uni} can be regarded as a generalization of \cite[Theorem 3]{JEF} from one equilibrium to multiple equilibria
because the weaker hypothesis {\bf(H1)} than (\ref{dapo}) allows (\ref{LS}) to have multiple equilibria as mentioned before.
In the aspect of the smoothness of systems, Theorem~\ref{uni} can be regarded as a generalization of
\cite[Theorems 2.1 and 2.2]{FL} from smooth Li\'enard systems to nonsmooth ones.

In the following theorem we give a sufficient condition for the existence of periodic annuli for system (\ref{LS}).
It can be looked as a sufficient condition for the nonexistence of isolated crossing periodic orbits, i.e.,
crossing limit cycles.

\begin{thm}
For system {\rm(\ref{LS})} with {\bf (H2)}, assume that {\rm(\ref{limit})} holds and
\begin{eqnarray}
\left.\frac{g^+(x)}{f^+(x)}\right|_{x=x^+(p)}\equiv\left.\frac{g^-(x)}{f^-(x)}\right|_{x=x^-(p)}
\label{noncon}
\end{eqnarray}
for all $p>0$. If system {\rm(\ref{LS})} has a crossing periodic orbit, then there exists a periodic annulus
including this crossing periodic orbit, i.e., there is no crossing limit cycles.
\label{nonexist}
\end{thm}

To apply our main results of system (\ref{LS}), we study the
number of crossing limit cycles for the piecewise linear system
\begin{eqnarray}
\dot z=\left\{
\begin{aligned}
&A^+z+b^+~~~~~~{\rm if}~~x>0,\\
&A^-z+b^-~~~~~~{\rm if}~~x<0,\\
\end{aligned}
\right.
\label{generalPLF}
\end{eqnarray}
where $z=(x, y)^\top\in\mathbb{R}^2$,
$$
A^\pm\!=\!\left(\!
\begin{array}{cc}
a^\pm_{11}&a^\pm_{12}\\
a^\pm_{21}&a^\pm_{22}
\end{array}
\!\right)\!\in\!\mathbb{R}^{2\times2},~~~~
b^\pm\!=\!\left(\!
\begin{array}{c}
b_1^\pm\\
b_2^\pm
\end{array}
\!\right)\!\in\!\mathbb{R}^2.
$$
In this paper we always assume that system (\ref{generalPLF}) is {\it nondegenerate}, i.e., $\det A^\pm\ne0$.
System (\ref{generalPLF}) has been widely used as a model in engineering,
physics and biology (see \cite{ChL-DB, ChL-EE, ATGW}), and many contributions have been made in recent years
(see \cite{ChL-CCJ, ChL-FEEE, ChL-FGG, ChL-EE1, ChL-SX, ChL-EEFF, HLH}).
Although the two subsystems of (\ref{generalPLF}) are linear, the switching of the vector fields
in different regions leads to great complexity and difficulty in
the research on the number of crossing limit cycles. When (\ref{generalPLF}) is continuous,
it is proved in \cite{ChL-EEFF} that there exists at most one crossing limit cycle and this number can be reached.
When (\ref{generalPLF}) is discontinuous, in many publications examples were provided for
(\ref{generalPLF}) to have three crossing limit cycles, such as
\cite{ChL-CCJ, ChL-FEEE, ChL-EE1, ChL-SX}.
However, the problem of maximum number of crossing limit cycles for discontinuous system (\ref{generalPLF})
is still open. On the other hand,
we checked in all papers presenting examples with three crossing limit cycles and found that all these systems
have {\it sliding sets}, namely $\{(0, y): (a^+_{12}y+b_1^+)(a^-_{12}y+b_1^-)<0\}\ne\emptyset$.
Therefore, a natural question is {\it how about the maximum number of crossing limit cycles
for discontinuous system {\rm(\ref{generalPLF})} without sliding sets}.
This question was answered in \cite{JEF, ChL-JJT}
for the case that $O$ is a $\Sigma$-monodromic singularity, i.e.,
all orbits in a small neighborhood of $O$ of (\ref{generalPLF}) turn around $O$, and the maximum number is $1$.
Besides, it was proved in \cite{EPPP} (resp. \cite{ChL-EEPPEE}) that the maximum number is also $1$ for focus-saddle type
(resp. for focus-focus type with partial regions of parameters).
However, the problem of the maximum number of crossing limit cycles for general discontinuous
(\ref{generalPLF}) without sliding sets is still open.
Applying our main results for system (\ref{LS}), we completely answer this open problem in the following theorem.
\begin{thm}
If discontinuous system {\rm(\ref{generalPLF})} has no sliding sets, then
there exists at most one crossing limit cycle and this number can be reached.
Moreover, it is possible that the number of equilibria surrounded by this crossing limit cycle is exactly $k$ {\rm(}$1\le k\le 3${\rm)}.
\label{linear}
\end{thm}

The remainder of this paper is organized as follows. In Section 2 we give proofs of
Theorems~\ref{suff}, \ref{uni} and \ref{nonexist} for nonsmooth Li\'enard system (\ref{LS}).
In Section 3 we apply the main results for system (\ref{LS}) to study the number of crossing limit cycles of system (\ref{generalPLF}) and
prove Theorem~\ref{linear}.

\section{Proofs of Theorems~\ref{suff}, \ref{uni} and \ref{nonexist}}
\setcounter{equation}{0}
\setcounter{lm}{0}
\setcounter{thm}{0}
\setcounter{rmk}{0}
\setcounter{df}{0}
\setcounter{cor}{0}

The purpose of this section is to provide the proofs of Theorems~\ref{suff}, \ref{uni} and \ref{nonexist}.
Firstly, we describe some geometrical properties of a crossing periodic orbit of system (\ref{LS})
with {\bf(H1)}.

\begin{figure}[htp]
  \begin{minipage}[t]{1.0\linewidth}
  \centering
  \includegraphics[width=2.50in]{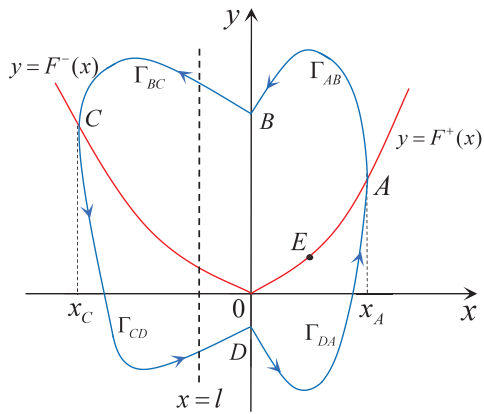}
  \end{minipage}
\caption{An illustration of geometrical properties of $\Gamma$.}
\label{illustra}
\end{figure}

\begin{lm}
If system {\rm(\ref{LS})} with {\bf(H1)} has a crossing periodic orbit $\Gamma$, then
\begin{description}
\setlength{\itemsep}{0mm}
\item {\rm(i)} $\Gamma$ intersects the curve $y=F^+(x)$ {\rm(}resp. $y=F^-(x)${\rm)} at a unique point in $x>0$ {\rm(}resp. $x<0${\rm)},
denoted by $A$ {\rm(}resp. $C${\rm)} as in {\rm Figure~\ref{illustra}};
\item {\rm(ii)} $\Gamma$ intersects any line $x=l$ satisfying $x_C<l<x_A$ at exactly two points, where $x_A$ and $x_C$ are the
abscissas of $A$ and $C$, respectively;
\item {\rm(iii)} $\Gamma$ surrounds $O$ and $E$ counterclockwise.
\end{description}
\label{geo}
\end{lm}

\begin{proof}
According to the third paragraph in Section 1, $O$ is an equilibrium when $g^+(0)g^-(0)\le0$ and
a parabolic fold-fold point when $g^+(0)g^-(0)>0$. This implies that $\Gamma$ cannot pass through $O$. Moreover,
$\dot x>0$ (resp. $<0$) for all $(x, y)$ in the below (resp. above) of $y=p(x)$, where $p(x)$ is defined in (\ref{pde}).
Therefore, conclusions (i) and (ii) hold and $\Gamma$ surrounds $O$ counterclockwise.
Let $B$ (resp. $D$) be the intersection of $\Gamma$ and the positive $y$-axis (resp. the negative $y$-axis)
and $\Gamma:=\Gamma_{AB}\cup\Gamma_{BC}\cup\Gamma_{CD}\cup\Gamma_{DA}$, see Figure~\ref{illustra}.
If $x_e=0$, conclusion (iii) obviously holds because $E$ lies at $O$.
If $x_e>0$, it follows from {\bf(H1)} that the curve corresponding to $\Gamma_{DA}$ (resp. $\Gamma_{AB}$) goes down
for $0< x< x_e$ and goes up for $x_e< x< x_A$ as $t$ increases.
Thus $\Gamma$ also surrounds $E$ counterclockwise, implying the conclusion (iii).
\end{proof}

\begin{proof}[{\bf Proof of Theorem~\ref{suff}}]
Under hypotheses, if system (\ref{LS}) has a crossing periodic orbit $\Gamma$, then it satisfies the geometrical properties in Lemma~\ref{geo}
and we still use the denotations in the proof of Lemma~\ref{geo}, see Figure~\ref{FxyFpy1}(a).
Conclusion (i) is obtained directly from Lemma~\ref{geo}(iii).

\begin{figure}[h]
  \begin{minipage}[t]{0.5\linewidth}
  \centering
  \includegraphics[width=2.60in]{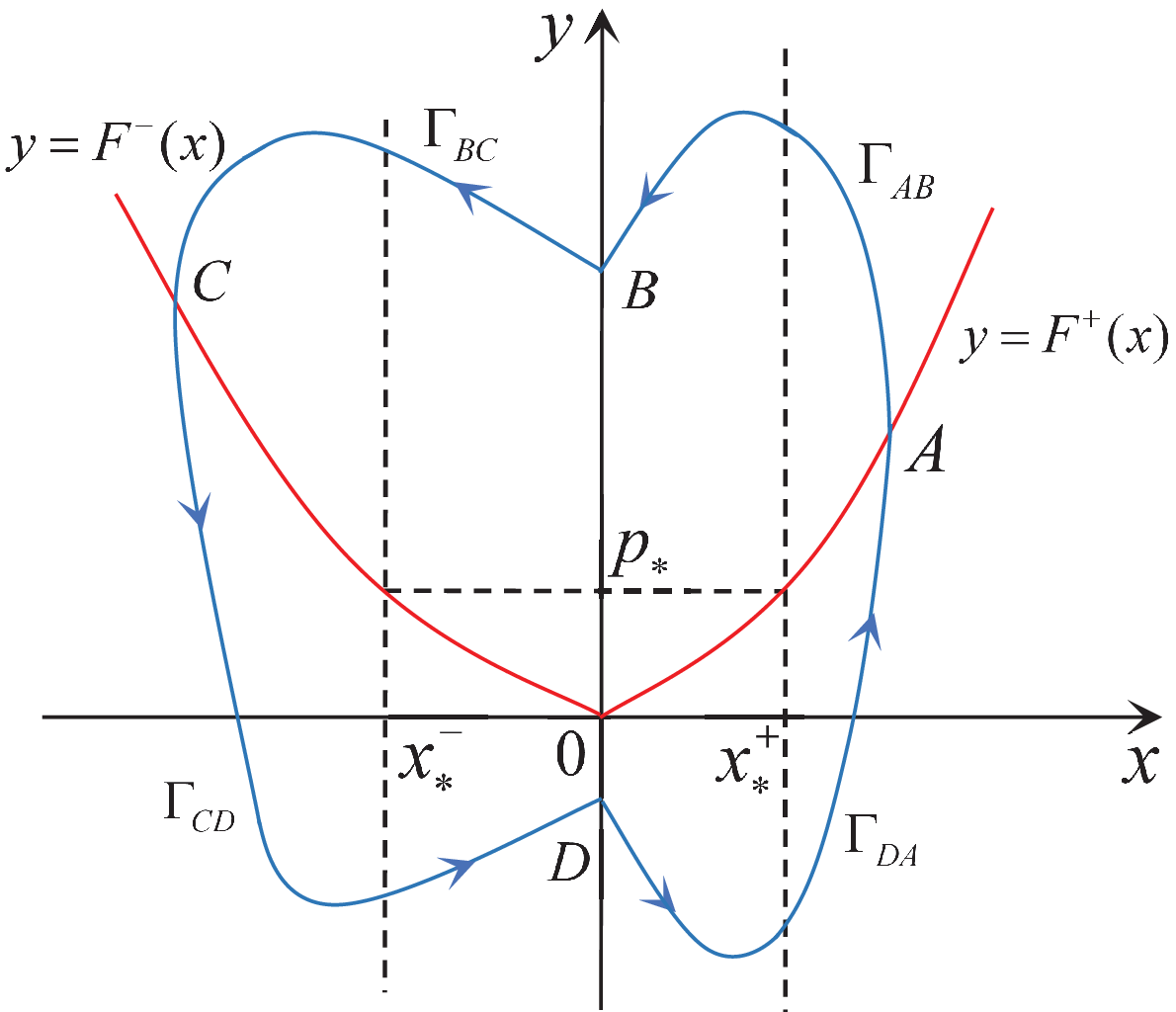}
  \caption*{(a)}
  \end{minipage}
  \begin{minipage}[t]{0.5\linewidth}
  \centering
  \includegraphics[width=2.1in]{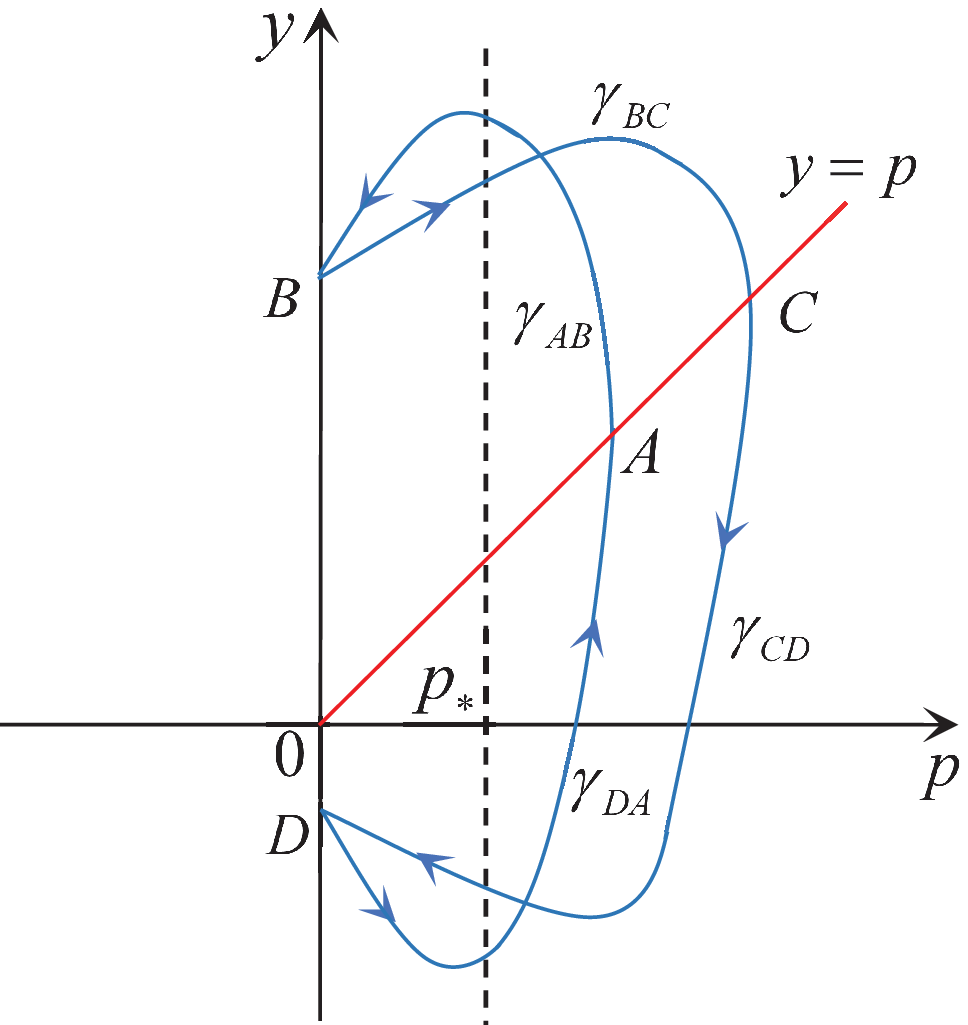}
   \caption*{(b)}
  \end{minipage}
\caption{$\Gamma$ in $xy$-plane and $\gamma$ in $py$-plane.}
\label{FxyFpy1}
\end{figure}

In order to prove conclusion (ii), we apply the change $p=p(x)$ to system (\ref{LS}), where $p(x)$ is defined in (\ref{pde}).
Thus the right and left systems of (\ref{LS}) are transformed into
\begin{eqnarray}
\left\{
\begin{aligned}
&\dot p=f^+(x^+(p))(p-y),\\
&\dot y=g^+(x^+(p)),
\end{aligned}
\right.
~~~~~~
\left\{
\begin{aligned}
&\dot p=f^-(x^-(p))(p-y),\\
&\dot y=g^-(x^-(p)),
\end{aligned}
\right.
\label{pyRR}
\end{eqnarray}
respectively, from which we get
\begin{eqnarray}
\frac{dy}{dp}=\frac{\varphi^+(p)}{p-y}:=\frac{g^+(x^+(p))}{f^+(x^+(p))(p-y)},~~~~~~
\frac{dy}{dp}=\frac{\varphi^-(p)}{p-y}:=\frac{g^-(x^-(p))}{f^-(x^-(p))(p-y)}
\label{pyR}
\end{eqnarray}
for $p>0$, respectively. Here $x^+(p)$ (resp. $x^-(p)$) given in (\ref{inverser}) (resp. (\ref{inversel}))
is the inverse function of $p=p(x)$ for $x\ge0$ (resp. $x<0$).
Clearly, (\ref{pyRR}) and (\ref{pyR}) can be continuously extended to $p=0$ due to the continuities
of $f^\pm, g^\pm$ at $x=0$ and {\bf(H3)}. Thus we always assume
that (\ref{pyRR}) and (\ref{pyR}) are well defined for $p\ge0$.
Moreover, under the change $p=p(x)$, the crossing periodic orbit $\Gamma$ becomes the orbit
$\gamma:=\gamma_{AB}\cup\gamma_{BC}\cup\gamma_{CD}\cup\gamma_{DA}$
in $py$-plane, see Figure~\ref{FxyFpy1}(b), where $\gamma_{DA}\cup\gamma_{AB}$ and $\gamma_{BC}\cup\gamma_{CD}$
are the orbits of the first system and the second one in (\ref{pyRR}), respectively.

Let $\Delta^+$ (resp. $\Delta^-$) be the region surrounded by $y$-axis and $\Gamma_{DA}\cup\Gamma_{AB}$ (resp. $\Gamma_{BC}\cup\Gamma_{CD}$),
$\Omega^+$ (resp. $\Omega^-$) be the region surrounded by $y$-axis and $\gamma_{DA}\cup\gamma_{AB}$ (resp. $\gamma_{BC}\cup\gamma_{CD}$),
$S(\Omega^\pm)$ be the areas of $\Omega^\pm$. By Green's formula we have
\begin{eqnarray}
\begin{aligned}
0=&\int_{\Gamma_{DA}\cup\Gamma_{AB}}\!\!-g^+(x)dx+(F^+(x)-y)dy+\int_{\Gamma_{BC}\cup\Gamma_{CD}}\!\!-g^-(x)dx+(F^-(x)-y)dy\\
=&\iint_{\Delta^+}f^+(x)dxdy+\iint_{\Delta^-}f^-(x)dxdy-\int_{\overline{BD}}-g^+(x)dx+(F^+(x)-y)dy\\
&-\int_{\overline{DB}}-g^-(x)dx+(F^-(x)-y)dy\\
=&\iint_{\Delta^+}f^+(x)dxdy+\iint_{\Delta^-}f^-(x)dxdy\\
=&\iint_{\Omega^+}dpdy-\iint_{\Omega^-}dpdy\\
=&S(\Omega^+)-S(\Omega^-).
\end{aligned}
\label{area}
\end{eqnarray}
Thus $\gamma_{BC}\cup\gamma_{CD}$ must cross $\gamma_{DA}\cup\gamma_{AB}$ at some $p\in(0, \min\{p_A, p_C\})$, where $p_A$ and $p_C$ are the abscissas of
points $A$ and $C$ in $py$-plane. Otherwise, from {\bf(H3)} we get that
$\gamma_{BC}$ (resp. $\gamma_{DA}$) always lies below $\gamma_{AB}$ (resp. $\gamma_{CD}$). This means that $S(\Omega^+)-S(\Omega^-)>0$,
contradicting (\ref{area}).

Suppose that $\varphi^+(p)<\varphi^-(p)$ for $0<p<\min\{p_A, p_C\}$, then
$y_{AB}(p)>y_{BC}(p)$ and $y_{CD}(p)>y_{DA}(p)$ for $0<p<\min\{p_A, p_C\}$ by applying the theory of differential inequalities to systems in (\ref{pyR}),
where $y=y_{AB}(p)$, $y=y_{BC}(p)$, $y=y_{CD}(p)$ and $y=y_{DA}(p)$ describe the orbits $\gamma_{AB}$, $\gamma_{BC}$, $\gamma_{CD}$ and $\gamma_{DA}$,
respectively. Thus $\gamma_{BC}\cup\gamma_{CD}$ does not cross $\gamma_{DA}\cup\gamma_{AB}$, implying a contradiction. Consequently,
$\varphi^+(p)=\varphi^-(p)$ has at least one solution in $0<p<\min\{p_A, p_C\}$, denoted by $p_*$.
Choosing $x^+_*:=x^+(p_*)$ and $x^-_*:=x^-(p_*)$, we finally obtain that the equations (\ref{eq})
have at least one solution $(x^-, x^+)=(x^-_*, x^+_*)$ with $x^-_*<0<x^+_*$ satisfying that $\Gamma$ transversally intersects
both the verticals $x=x^\pm_*$, i.e., conclusion (ii) is proved.
\end{proof}

\begin{proof}[{\bf Proof of Theorem~\ref{uni}}]
The essential idea of this proof comes from \cite{FL, JEF} and it is accomplished by two steps.
Assume that system (\ref{LS}) has a crossing periodic orbit $\Gamma:=(x(t), y(t))$ and define
$$\lambda_\Gamma:=\int_{\Gamma^-}f^-(x(t))dt+\int_{\Gamma^+}f^+(x(t))dt,$$
where $\Gamma^+:=\Gamma\cap\{(x, y):x\ge0\}$ and $\Gamma^-:=\Gamma\cap\{(x, y):x\le0\}$.
In the first step, we prove $\lambda_\Gamma<0$, which implies that $\Gamma$ is a stable and hyperbolic crossing limit cycle
by \cite[Theorem 2.1]{DZ}. In the second step, we prove that system (\ref{LS}) cannot have two stable and hyperbolic limit cycles
in succession, which implies the uniqueness of crossing limit cycles associated with the result of the first step.

{\bf Step 1}. {\it We prove $\lambda_\Gamma<0$}.

Following the denotations and geometric properties of $\Gamma$ in Lemma~\ref{geo} and Theorem~\ref{suff},
we firstly prove $p_C>p_A$. In fact, under the assumption of theorem,
$\varphi^+(p)=\varphi^-(p)$ has a unique solution $p_*\in(0, \min\{p_A, p_C\})$. Moreover, $\varphi^+(p)<\varphi^-(p)$ for $0<p<p_*$ and $\varphi^+(p)>\varphi^-(p)$ for $p>p_*$. Applying the theory of differential inequalities to systems in (\ref{pyR}), we obtain
\begin{eqnarray}
y_{AB}(p)>y_{BC}(p),~~~~~y_{CD}(p)>y_{DA}(p)~~~~~~~~~~{\rm for}~~0<p<p_*.
\label{ewuihfkjn}
\end{eqnarray}
From the proof of Theorem~\ref{suff} $\gamma_{BC}\cup\gamma_{CD}$ must cross $\gamma_{DA}\cup\gamma_{AB}$. Without loss of generality, assume that $\gamma_{BC}\cup\gamma_{CD}$ and $\gamma_{DA}\cup\gamma_{AB}$ have a crossing point at a value $p_1$ for which $\gamma_{AB}$ crosses $\gamma_{BC}$.
Then, it follows from (\ref{ewuihfkjn}) and $\varphi^+(p)>\varphi^-(p)$ for $p>p_*$ that $p_*\le p_1<\min\{p_A, p_C\}$ and
\begin{eqnarray}
\begin{aligned}
&y_{AB}(p)>y_{BC}(p)~~~~~~~~{\rm for}~~0<p<p_1,\\
&y_{AB}(p)<y_{BC}(p)~~~~~~~~{\rm for}~~p_1<p<\min\{p_A, p_C\}.
\end{aligned}
\label{uiwhfvn}
\end{eqnarray}
Here the theory of differential inequalities is applied in the second inequality. Moreover, since the number of crossing points of  $\gamma_{BC}\cup\gamma_{CD}$ and $\gamma_{DA}\cup\gamma_{AB}$ must be even from (\ref{ewuihfkjn}),
we further obtain that there exists a value $p_2$ with $p_*\le p_2<\min\{p_A, p_C\}$ such that $\gamma_{CD}$ crosses $\gamma_{DA}$ at $p_2$. Similarly,
\begin{eqnarray}
\begin{aligned}
&y_{CD}(p)>y_{DA}(p)~~~~~~~~{\rm for}~~0<p<p_2,\\
&y_{CD}(p)<y_{DA}(p)~~~~~~~~{\rm for}~~p_2<p<\min\{p_A, p_C\}.
\end{aligned}
\label{uefnkj}
\end{eqnarray}
Hence, combining with (\ref{uiwhfvn}) and (\ref{uefnkj}), we get $p_C\ge p_A$.

On the other hand, we have $\varphi^+(p)>\varphi^-(p)>0$ and $0>p-y_{AB}(p)>p-y_{BC}(p)$ for $p_*\ll p<p_A$. Thus
$$
\begin{aligned}
\frac{dy_{AB}(p)-dy_{BC}(p)}{dp}&=\frac{\varphi^+(p)}{p-y_{AB}(p)}-\frac{\varphi^-(p)}{p-y_{BC}(p)}\\
&<\frac{\varphi^+(p)}{p-y_{BC}(p)}-\frac{\varphi^-(p)}{p-y_{BC}(p)}\\
&=\frac{\varphi^+(p)-\varphi^-(p)}{p-y_{BC}(p)}\\
&<0
\end{aligned}
$$
for all $p$ with $p_*\ll p<p_A$, so that $y_{AB}(p)-y_{BC}(p)$ is strictly decreasing in $p_*\ll p<p_A$ and then $p_C\ne p_A$, i.e.,
$p_C>p_A$.

\begin{figure}[htp]
  \begin{minipage}[t]{0.5\linewidth}
  \centering
  \includegraphics[width=2.60in]{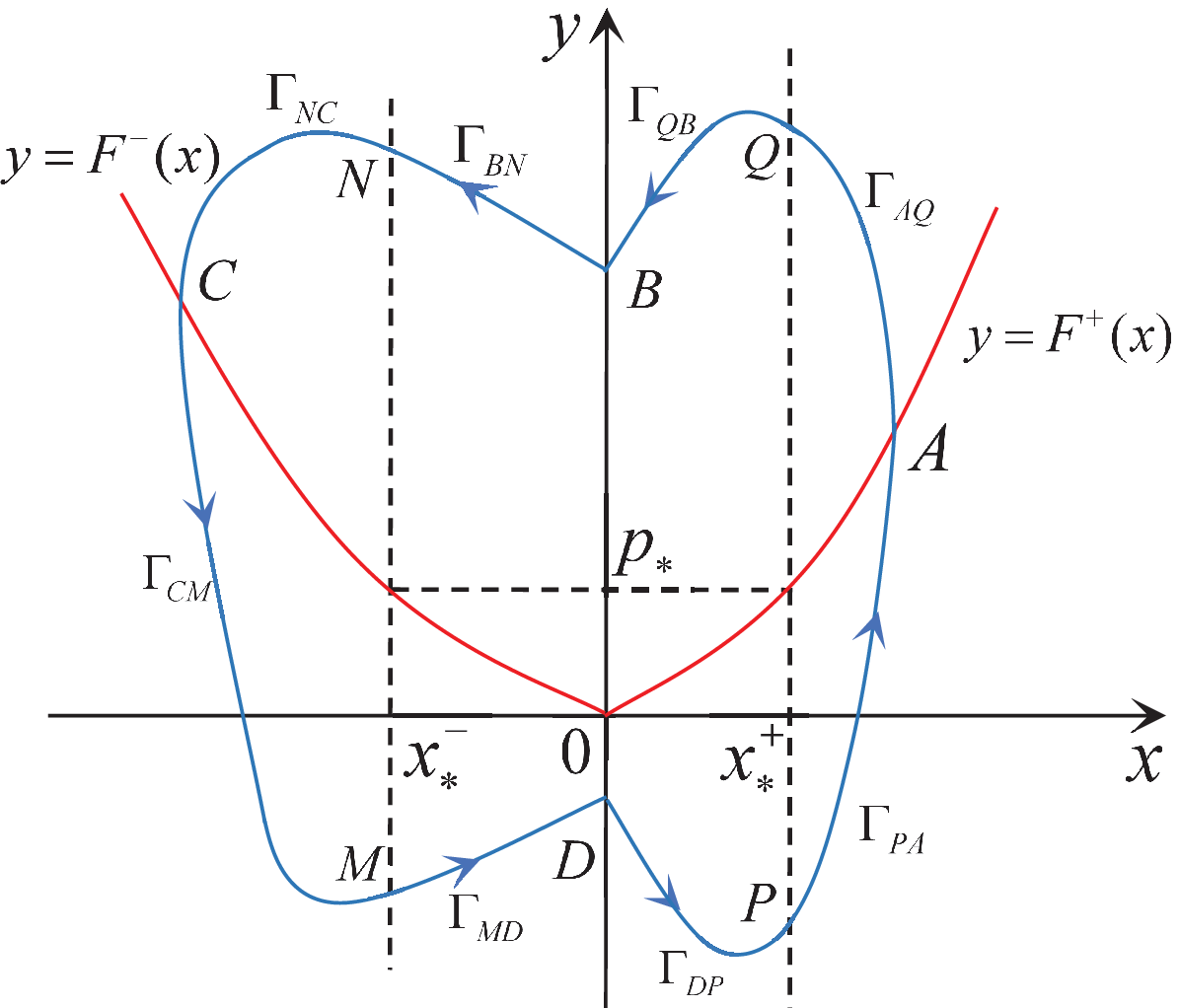}
  \caption*{(a)}
  \end{minipage}
  \begin{minipage}[t]{0.5\linewidth}
  \centering
  \includegraphics[width=2.1in]{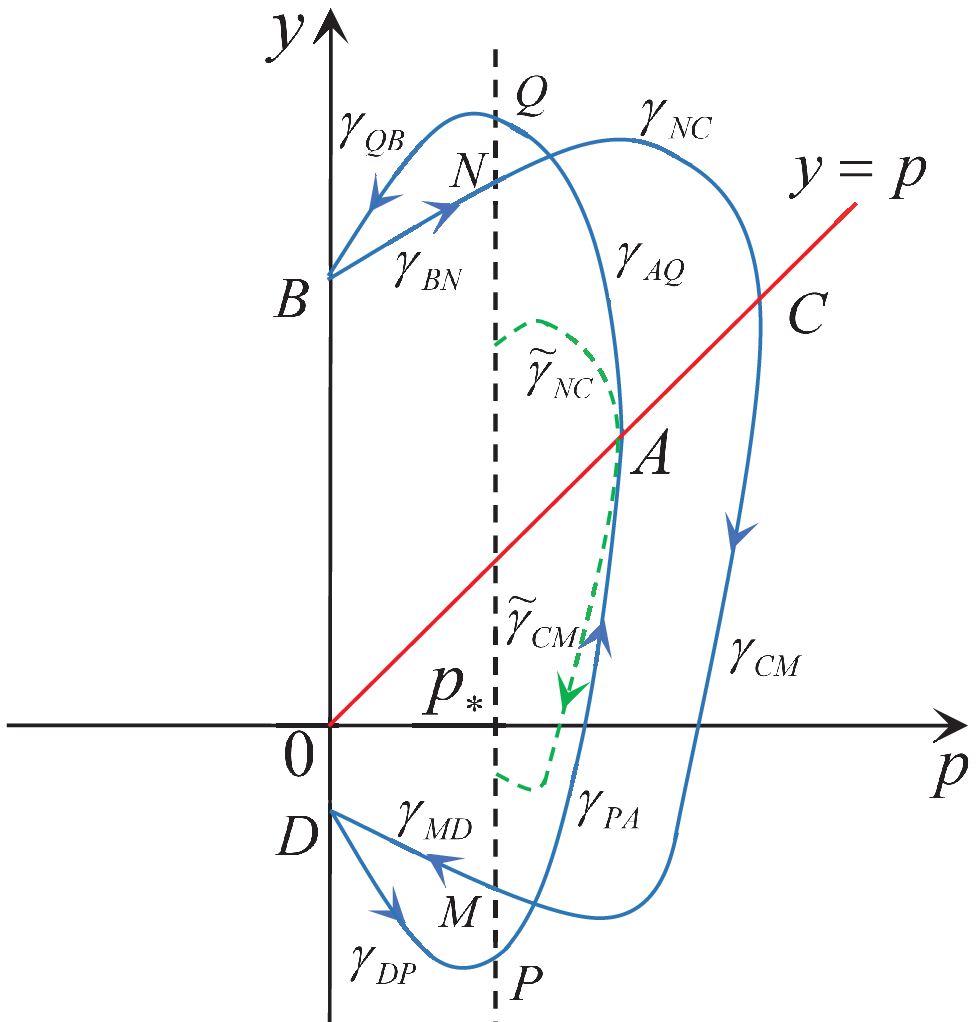}
   \caption*{(b)}
  \end{minipage}
\caption{The eight orbit arcs of $\Gamma$ in $xy$-plane and the ones of $\gamma$ in $py$-plane.}
\label{Fpy2}
\end{figure}

Let $M$ and $N$ (resp. $P$ and $Q$) be the intersections of $\Gamma$ and the vertical $x=x_*^-$ (resp. $x=x_*^+$). Then we denote $\Gamma$
by
$$\Gamma=\Gamma_{DP}\cup\Gamma_{PA}\cup\Gamma_{AQ}\cup\Gamma_{QB}\cup\Gamma_{BN}\cup\Gamma_{NC}\cup\Gamma_{CM}\cup\Gamma_{MD}$$
as shown in Figure~\ref{Fpy2}(a) and $\gamma$ which corresponds with $\Gamma$ under the change $p=p(x)$ by
$$\gamma=\gamma_{DP}\cup\gamma_{PA}\cup\gamma_{AQ}\cup\gamma_{QB}\cup\gamma_{BN}\cup\gamma_{NC}\cup\gamma_{CM}\cup\gamma_{MD}$$
as shown in Figure~\ref{Fpy2}(b), where $\gamma_{DP}\cup\gamma_{PA}\cup\gamma_{AQ}\cup\gamma_{QB}$ and $\gamma_{BN}\cup\gamma_{NC}\cup\gamma_{CM}\cup\gamma_{MD}$
are the orbits of the first system and the second one in (\ref{pyRR}), respectively.
Therefore,
\begin{eqnarray}
\begin{aligned}
\lambda_\Gamma&=\int_{\Gamma_{DP}\cup\Gamma_{QB}}\!f^+(x)dt\!+\!\int_{\Gamma_{BN}\cup\Gamma_{MD}}\!f^-(x)dt
\!+\!\int_{\Gamma_{PA}\cup\Gamma_{AQ}}\!f^+(x)dt\!+\!\int_{\Gamma_{NC}\cup\Gamma_{CM}}\!f^-(x)dt\\
&=\int_{\gamma_{DP}\cup\gamma_{QB}}\!\frac{dp}{p-y}+\int_{\gamma_{BN}\cup\gamma_{MD}}\!\frac{dp}{p-y}
+\int_{\gamma_{PA}\cup\gamma_{AQ}}\!\frac{dp}{p-y}+\int_{\gamma_{NC}\cup\gamma_{CM}}\!\frac{dp}{p-y}
\end{aligned}
\label{lambdaex}
\end{eqnarray}
due to $\dot p=f^\pm(x^\pm(p))(p-y)$ in (\ref{pyRR}). For brevity, we neglect the variable $t$ of $x, y$ and $p$
in (\ref{lambdaex}) and the rest of this proof if confusion does not arise.

Firstly, we prove
\begin{eqnarray}
J_1:=\int_{\gamma_{DP}\cup\gamma_{QB}}\!\frac{dp}{p-y}+\int_{\gamma_{BN}\cup\gamma_{MD}}\!\frac{dp}{p-y}<0.
\label{firstpart}
\end{eqnarray}
Let $y=y_{DP}(p)$, $y=y_{QB}(p), y=y_{BN}(p)$ and $y=y_{MD}(p)$ for $0<p\le p_*$ describe $\gamma_{DP}$, $\gamma_{QB}, \gamma_{BN}$ and $\gamma_{MD}$, respectively. Then
$$p-y_{DP}(p)>0, ~~~p-y_{MD}(p)>0, ~~~p-y_{BN}(p)<0, ~~~p-y_{QB}(p)<0.$$
Moreover, from (\ref{uiwhfvn}) and (\ref{uefnkj}) we have
\begin{eqnarray}
y_{BN}(p)<y_{QB}(p),~~~~~~~~~y_{MD}(p)>y_{DP}(p)
\label{ahfj}
\end{eqnarray}
for $0<p<p_*$.
Hence,
$$
\begin{aligned}
J_1=&\int_0^{p_*}\frac{dp}{p-y_{DP}(p)}+\int_{p_*}^0\frac{dp}{p-y_{QB}(p)}+\int_0^{p_*}\frac{dp}{p-y_{BN}(p)}+\int_{p_*}^0\frac{dp}{p-y_{MD}(p)}\\
=&\int_0^{p_*}\frac{y_{DP}(p)-y_{MD}(p)}{(p-y_{DP}(p))(p-y_{MD}(p))}dp+\int_0^{p_*}\frac{y_{BN}(p)-y_{QB}(p)}{(p-y_{BN}(p))(p-y_{QB}(p))}dp\\
<&0,
\end{aligned}
$$
i.e., (\ref{firstpart}) holds.

Secondly, we prove
\begin{eqnarray}
J_2:=\int_{\gamma_{PA}\cup\gamma_{AQ}}\!\frac{dp}{p-y}+\int_{\gamma_{NC}\cup\gamma_{CM}}\!\frac{dp}{p-y}<0,
\label{secondpart}
\end{eqnarray}
implying $\lambda_\Gamma=J_1+J_2<0$ from (\ref{lambdaex}) and (\ref{firstpart}).
To do this, we define
$$
\mu:=\frac{p_A-p_*}{p_C-p_*}, ~~~~~~~\eta:=\frac{(p_C-p_A)p_*}{p_C-p_*}
$$
as in \cite{FL}. Clearly, $0<\mu<1$ due to $p_C>p_A>p_*$ and $\eta=(1-\mu)p_*$.
By the linear transformation
$$
\left\{
\begin{aligned}
\widetilde p&=\mu p+\eta:=\psi(p),\\
\widetilde y&=\mu y+\eta:=\phi(y),
\end{aligned}
\right.
$$
the second system in (\ref{pyRR}) is transformed into
\begin{eqnarray}
\left\{
\begin{aligned}
\dot{\widetilde p}&=f^-\left(x^-\left(\frac{\widetilde p-\eta}{\mu}\right)\right)(\widetilde p-\widetilde y),\\
\dot{\widetilde y}&=\mu g^-\left(x^-\left(\frac{\widetilde p-\eta}{\mu}\right)\right).
\end{aligned}
\right.
\label{zfasdada}
\end{eqnarray}
Denote the orbit of (\ref{zfasdada}) corresponding with $\gamma_{NC}\cup\gamma_{CM}$ by $\widetilde\gamma_{NC}\cup\widetilde\gamma_{CM}$.
Since $\psi(p_C)=p_A$, $\phi(p_C)=p_A$, $\psi(p_*)=p_*$ and $\phi(p_*)=p_*$, the orbit $\widetilde\gamma_{NC}\cup\widetilde\gamma_{CM}$
is from $(p_*, \mu y_{NC}(p_*)+\eta)$ to $(p_*, \mu y_{CM}(p_*)+\eta)$ after passing through $A$, see Figure~\ref{Fpy2}(b).
Thus
\begin{eqnarray}
\begin{aligned}
J_2=&\int_{\gamma_{PA}\cup\gamma_{AQ}}\frac{dp}{p-y}+\int_{\widetilde\gamma_{NC}\cup\widetilde\gamma_{CM}}\frac{d\widetilde p}{\widetilde p-\widetilde y}\\
=&\int_{p_*}^{p_A}\frac{dp}{p-y_{PA}(p)}+\int_{p_A}^{p_*}\frac{dp}{p-y_{AQ}(p)}+
\int_{p^*}^{p_A}\frac{dp}{p-\widetilde y_{NC}(p)}+\int_{p_A}^{p^*}\frac{dp}{ p-\widetilde y_{CM}(p)}\\
=&\int_{p^*}^{p_A}\frac{y_{PA}(p)-\widetilde y_{CM}(p)}{(p-y_{PA}(p))(p-\widetilde y_{CM}(p))}dp+
\int_{p^*}^{p_A}\frac{\widetilde y_{NC}(p)-y_{AQ}(p)}{(p-y_{AQ}(p))(p-\widetilde y_{NC}(p))}dp,
\end{aligned}
\label{saa}
\end{eqnarray}
where $y=\widetilde y_{NC}(p)$ and $y=\widetilde y_{CM}(p)$ describe $\widetilde\gamma_{NC}$ and $\widetilde\gamma_{CM}$, respectively.
In order to prove $J_2<0$, from (\ref{saa}) it is sufficient to prove
\begin{eqnarray}
y_{PA}(p)-\widetilde y_{CM}(p)<0,~~~~~~~~~\widetilde y_{NC}(p)-y_{AQ}(p)<0
\label{uhfb}
\end{eqnarray}
for $p_*\le p<p_A$ because
$p-y_{PA}(p)>0, p-\widetilde y_{CM}(p)>0, p-y_{AQ}(p)<0, p-\widetilde y_{NC}(p)<0$.
Clearly, $\phi(y)-y=(1-\mu)(p_*-y)$, $p_*-y_{CM}(p_*)>0$ and $p_*-y_{NC}(p_*)<0$. Thus
\begin{eqnarray}
\begin{aligned}
&\widetilde y_{CM}(p_*)-y_{CM}(p_*)=\phi(y_{CM}(p_*))-y_{CM}(p_*)=(1-\mu)(p_*-y_{CM}(p_*))>0,\\
&\widetilde y_{NC}(p_*)-y_{NC}(p_*)=\phi(y_{NC}(p_*))-y_{NC}(p_*)=(1-\mu)(p_*-y_{NC}(p_*))<0
\end{aligned}
\label{nskl}
\end{eqnarray}
due to $0<\mu<1$.
Using (\ref{ahfj}) and (\ref{nskl}), we get
\begin{eqnarray}
\begin{aligned}
&y_{PA}(p_*)=y_{DP}(p_*)\le y_{MD}(p_*)=y_{CM}(p_*)<\widetilde y_{CM}(p_*),\\
&y_{AQ}(p_*)=y_{QB}(p_*)\ge y_{BN}(p_*)=y_{NC}(p_*)>\widetilde y_{NC}(p_*),
\end{aligned}
\label{kjsenjnvd}
\end{eqnarray}
i.e., (\ref{uhfb}) holds for $p=p_*$.

To prove (\ref{uhfb}) for $p_*<p<p_A$, we consider system (\ref{zfasdada}) without tildes and the first system in (\ref{pyR}) for $p_*<p<p_A$.
We can rewrite them as
\begin{eqnarray}
\frac{dy}{dp}=h(p, y):=\frac{\varphi^-((p-\eta)/\mu)}{(p-\eta)/\mu}\cdot\frac{p-\eta}{p-y}
\label{zf2}
\end{eqnarray}
and
\begin{eqnarray}
\frac{dy}{dp}=H(p, y):=\frac{\varphi^+(p)}{p}\cdot\frac{p}{p-y},
\label{ajc2}
\end{eqnarray}
respectively.
We only prove the first inequality in (\ref{uhfb}) when {\bf(H4)} or {\bf(H5)} holds, and
the second one can be treated analogously. Thus $p-y>0$ is always assumed in the following.

Assume that system (\ref{LS}) satisfies {\bf(H4)}. Then $\varphi^+(p)/p$ is decreasing in $p_*<p<p_A$. Moreover,
since $0<\mu<1$ and $\eta=(1-\mu)p_*>0$, we have $(p-\eta)/\mu>p$ for $p>p_*$. Thus
\begin{eqnarray}
\frac{\varphi^-((p-\eta)/\mu)}{(p-\eta)/\mu}<\frac{\varphi^+((p-\eta)/\mu)}{(p-\eta)/\mu}\le\frac{\varphi^+(p)}{p}
\label{uwerc}
\end{eqnarray}
due to $\varphi^-(p)<\varphi^+(p)$ for $p>p_*$. According to {\bf(H4)}, we get $x_*^+\ge x_e$,
so that $g^+(x)/f^+(x)>0$ for $x>x_*^+$ by {\bf(H1)} and {\bf(H2)}, i.e.,
$\varphi^+(p)>0$ for $p_*<p<p_A$. Hence, for $p_*<p<p_A$ and $p-y>0$, $H(p, y)>0$ and then
\begin{eqnarray}
h(p, y)<\frac{\varphi^+(p)}{p}\cdot\frac{p-\eta}{p-y}=H(p, y)\cdot\frac{p-\eta}{p}<H(p, y),
\label{nvks}
\end{eqnarray}
where (\ref{uwerc}) and the fact that $p>\eta>0$ are used. Since $y_{PA}(p_*)<\widetilde y_{CM}(p_*)$ as in (\ref{kjsenjnvd}), if
there exists $\bar p$ with $p_*<\bar p<p_A$ such that $y_{PA}(\bar p)=\widetilde y_{CM}(\bar p)$, we obtain from (\ref{nvks}) that
$y_{PA}(p)>\widetilde y_{CM}(p)$ for $\bar p<p\le p_A$ by applying the theory of differential inequalities to systems (\ref{zf2}) and (\ref{ajc2}). This contradicts the fact that $y_{PA}(p_A)=\widetilde y_{CM}(p_A)$, and consequently, $y_{PA}(p)<\widetilde y_{CM}(p)$ for $p_*<p<p_A$, i.e., the first inequality of (\ref{uhfb}) holds under {\bf(H4)}.

Now assume that system (\ref{LS}) satisfies {\bf(H5)}. Considering the function
$$G(p):=\mu\varphi^-\left(\frac{p-\eta}{\mu}\right)-\varphi^+(p)$$
for $p_*<p<p_A$, we obtain
$$G'(p)=K^-\left(x^-\left(\frac{p-\eta}{\mu}\right)\right)-K^+(x^+(p))<0$$
by {\bf(H5)} and $(p-\eta)/\mu>p$.
Moreover, since $\eta=(1-\mu)p_*$ and $\varphi^+(p_*)=\varphi^-(p_*)$,
$$G(p_*)=\mu\varphi^-\left(\frac{p_*-\eta}{\mu}\right)-\varphi^+(p_*)=\mu\varphi^-(p_*)-\varphi^+(p_*)=(\mu-1)\varphi^+(p_*).$$
Combining with $\varphi^+(p_*)=g^+(x_*^+)/f^+(x_*^+)$ and $0<\mu<1$, we further have $G(p_*)\le0$ if $x_*^+\ge x_e$ and $G(p_*)>0$ if $0<x_*^+<x_e$ by
{\bf(H1)} and {\bf(H2)}. In the first case, $G(p)<0$ and then
$$h(p, y)-H(p, y)=\frac{G(p)}{p-y}<0$$
for $p_*<p<p_A$ and $p-y>0$, implying that $y_{PA}(p)-\widetilde y_{CM}(p)<0$ for $p_*<p<p_A$ by a same analysis with the last paragraph.
In the second case, $G(p)$ has at most one zero point in $p_*<p<p_A$.
When $G(p)$ has no zero points, $G(p)>0$ for $p_*<p<p_A$, so that $h(p, y)-H(p, y)>0$ due to $p-y>0$. By the theory of differential inequalities,
it directly follows from $y_{PA}(p_*)<\widetilde y_{CM}(p_*)$ that $y_{PA}(p)-\widetilde y_{CM}(p)<0$ for $p_*<p<p_A$.
When $G(p)$ has a zero point, denoted by $q$, we get $h(p, y)-H(p, y)>0$ for $p_*<p<q$ and
$h(p, y)-H(p, y)<0$ for $q<p<p_A$. Thus $y_{PA}(p)-\widetilde y_{CM}(p)<0$ for $p_*<p\le q$ and, by a same analysis with the last paragraph
$y_{PA}(p)-\widetilde y_{CM}(p)<0$ for $q<p<p_A$. In conclusion, the first inequality of (\ref{uhfb}) also holds under {\bf(H5)}.

{\bf Step 2}. {\it We prove the uniqueness of crossing periodic orbits.}

Assume that system (\ref{LS}) has two adjacent crossing periodic orbits $\Gamma_1$ and $\Gamma_2$.
By Theorem~\ref{suff}, both $\Gamma_1$ and $\Gamma_2$ surround $O$ and $E$. Moreover,
it follows from Step 1 and \cite[Theorem 2.1]{DZ} that both $\Gamma_1$ and $\Gamma_2$ are stable and hyperbolic crossing limit cycles.
Let $\mathcal{A}$ be the open region surrounded by $\Gamma_1$ and $\Gamma_2$. Consider the $\alpha$-limit set $L$ of the orbit of (\ref{LS}) with some initial value $(x_0, y_0)\in\mathcal{A}$ having $\Gamma_1$ as the $\omega$-limit set. Similar to the smooth case \cite{FL},
by the Poincar\'e-Bendixson Theorem in nonsmooth dynamical systems (see \cite{CBTCR})
and the special structure of (\ref{LS}), $L$ must consist of an equilibrium $(\bar x, p(\bar x))\in\mathcal{A}$ and an unstable homoclinic orbit to $(\bar x, p(\bar x))$ which cuts the switching line $x$-axis. On the other hand, since {\bf(H1)} holds and all crossing periodic orbits surround $E$, we get $\bar x<0$, i.e., $(\bar x, p(\bar x))$ is an equilibrium of the left system. Thus, from {\bf(H2)} we have
${\rm div}(F^-(x)-y, g^-(x))|_{x=\bar x}=f^-(\bar x)<0$. This contradicts that $L$ is unstable by \cite[Theorem 1]{CSDZD}, which not only holds for hyperbolic saddles, in fact, but also holds for semi-hyperbolic ones. That is, $L$ cannot be $\alpha$-limit set of the orbit of (\ref{LS}) with the initial value $(x_0, y_0)$.
Finally, we conclude that (\ref{LS}) has at most one crossing periodic orbit.

Combining with Steps 1 and 2, we complete the proof, that is, system (\ref{LS}) has at most one crossing periodic orbit, which is a
stable and hyperbolic crossing limit cycle if it exists.
\end{proof}

\begin{proof}[{\bf Proof of Theorem~\ref{nonexist}}]
As in the proof of Theorem~\ref{suff}, by the change $p=p(x)$ we can transform the right and left systems of (\ref{LS}) into
the systems in (\ref{pyRR}) and then get differential equations in (\ref{pyR})
for $p>0$. The differential equations in (\ref{pyR}) can be
continuously extended to $p=0$ by defining $\varphi^\pm(0)=\eta^\pm$ due to (\ref{limit}).
Additionally, from (\ref{noncon}) we have $\eta^+=\eta^-$ and $\varphi^+(p)\equiv\varphi^-(p)$ for all $p\ge0$, implying
that the two equations in (\ref{pyR}) coincide for $p\ge 0$. Hence, any orbit of the first differential equations in (\ref{pyR})
going from a point in the negative $y$-axis to a point in the positive $y$-axis corresponds with
a crossing periodic orbit of (\ref{LS}). Conversely, the existence of crossing periodic orbits
of (\ref{LS}) ensures that the first differential equations in (\ref{pyR}) have an orbit going
from a point in the negative $y$-axis to a point in the positive $y$-axis.
Consequently, if (\ref{LS}) has a crossing periodic orbit $\Gamma$, then all orbits in the neighborhood of $\Gamma$ are crossing periodic orbits by the continuous dependence of solutions on initial values, i.e., the proof is completed.
\end{proof}

\section{Application to discontinuous piecewise linear systems}
\setcounter{equation}{0}
\setcounter{lm}{0}
\setcounter{thm}{0}
\setcounter{rmk}{0}
\setcounter{df}{0}
\setcounter{cor}{0}
In this section we apply Theorems~\ref{suff}, \ref{uni} and \ref{nonexist} to study the number of crossing limit cycles for discontinuous system (\ref{generalPLF}). In particular, the proof of Theorem~\ref{linear} will be presented.

According to \cite{ChL-EE}, system (\ref{generalPLF}) has no crossing limit cycles for $a_{12}^+a_{12}^-\le0$
because the $x$-component of both vector fields has same sign on crossing sets. For $a_{12}^+a_{12}^->0$,
it is proved in \cite[Proposition 3.1]{ChL-EE} that (\ref{generalPLF}) is $C^0$-homeomorphic to the Li\'enard canonical form
\begin{eqnarray}
\left(
\begin{array}{c}
\dot x\\
\dot y
\end{array}
\right)=
\left\{
\begin{aligned}
\left(
\begin{array}{cc}
t_R&-1\\
d_R&0
\end{array}
\right)\left(
\begin{array}{c}
x\\
y
\end{array}
\right)-
\left(
\begin{array}{c}
-b\\
a_R
\end{array}
\right)~~~~~~~~{\rm if}~~x>0,\\
\left(
\begin{array}{cc}
t_L&-1\\
d_L&0
\end{array}
\right)\left(
\begin{array}{c}
x\\
y
\end{array}
\right)-
\left(
\begin{array}{c}
0\\
a_L
\end{array}
\right)~~~~~~~~{\rm if}~~x<0,
\end{aligned}
\right.
\label{PPLF}
\end{eqnarray}
where $t_{\{R,L\}}$ and $d_{\{R,L\}}$ are the traces and determinants of $A^\pm$,
$$
b=\frac{a_{12}^-}{a_{12}^+}b_1^+-b_1^-,~~
~~a_L=a_{12}^-b_2^-\!-\!a_{22}^-b_1^-,~~~~
a_R=\frac{a_{12}^-}{a_{12}^+}(a_{12}^+b_2^+\!-\!a_{22}^+b_1^+).
$$
Although (\ref{generalPLF}) and (\ref{PPLF}) are not $\Sigma$-equivalent,
there exists a topological equivalence for all their orbits without sliding segments
as indicated in \cite{ChL-EE}. This means that
crossing limit cycles of (\ref{generalPLF}) are transformed into crossing
limit cycles of (\ref{PPLF}) in a homeomorphic way. Therefore, in order to study the existence, uniqueness and number of crossing limit cycles
of (\ref{generalPLF}), we only need to consider (\ref{PPLF}).

\begin{thm}
Assume that system {\rm(\ref{PPLF})} satisfies
\begin{eqnarray}
b=0,~~~ t_L<0,~~~t_R>0,~~~ d_L>0,~~~ d_R>0.
\label{atcon}
\end{eqnarray}
\begin{description}
\setlength{\itemsep}{0mm}
\item[]{\rm(i)} If $a_R/t_R>a_L/t_L$, then a necessary condition for the existence of crossing periodic orbits is
$d_R/t_R^2>d_L/t_L^2$. In addition, if there exists a crossing periodic orbit, then it is unique and stable.
\item[]{\rm(ii)} If $a_R/t_R<a_L/t_L$, then a necessary condition for the existence of crossing periodic orbits is
$d_R/t_R^2<d_L/t_L^2$. In addition, if there exists a crossing periodic orbit, then it is unique and unstable.
\item[]{\rm(iii)} If $a_R/t_R=a_L/t_L$, then a necessary condition for the existence of crossing periodic orbits is
$d_R/t_R^2=d_L/t_L^2$. In addition, if there exists a crossing periodic orbit, then there exists a
periodic annulus including this crossing periodic orbit.
\end{description}
\label{existandunique}
\end{thm}

\begin{proof}
Since $b=0$, system (\ref{PPLF}) is exactly the nonsmooth Li\'enard system (\ref{LS}) satisfying
\begin{eqnarray}
\begin{aligned}
&F^+(x)=t_Rx,~~~~~~~~&&f^+(x)=t_R,~~~~~~~~~g^+(x)=d_Rx-a_R,\\
&F^-(x)=t_Lx,~~~~~~~~&&f^-(x)=t_L,~~~~~~~~~g^-(x)=d_Lx-a_L.
\end{aligned}
\label{expres}
\end{eqnarray}
Clearly, it follows from (\ref{expres}) and $d_R>0$ in (\ref{atcon}) that {\bf(H1)} holds by choosing
$x_e:=0$ if $a_R\le0$ and $x_e:=a_R/d_R$ if $a_R>0$.
Moreover, {\bf(H2)} holds because of (\ref{expres}) and $t_R>0>t_L$ in (\ref{atcon}).
By the definitions of $x^+(p)$ and $x^-(p)$ given below {\bf(H2)}, we get
$x^+(p)=p/t_R$ and $x^-(p)=p/t_L$. Moreover, for system (\ref{PPLF}) the equations (\ref{eq}) become
\begin{eqnarray}
t_Lx^-=t_Rx^+, ~~~~~~~~~~
\frac{d_Lx^--a_L}{t_L}=\frac{d_Rx^+-a_R}{t_R}.
\label{induceeq}
\end{eqnarray}

If $a_R/t_R>a_L/t_L$, then
$$\lim_{p\rightarrow0^+}\frac{g^+(x^+(p))}{f^+(x^+(p))}=-\frac{a_R}{t_R}<-\frac{a_L}{t_L}=\lim_{p\rightarrow0^+}\frac{g^-(x^-(p))}{f^-(x^-(p))},$$
i.e., {\bf(H3)} holds.
Thus, by Theorem~\ref{suff} a necessary condition for the existence of crossing periodic
orbits is that the equations (\ref{induceeq})
have solutions with $x^-<0<x^+$, which is equivalent to $d_R/t_R^2>d_L/t_L^2$ because $a_R/t_R>a_L/t_L$ and $t_R>0$.
On the other hand, if (\ref{PPLF}) has a crossing periodic orbit, then
$$K^-(x^-(p_2))=\frac{d_L}{t_L^2}<\frac{d_R}{t_R^2}=K^+(x^+(p_1))$$
for all $p_1, p_2$ satisfying $p_2>p_1>0$, i.e., {\bf(H5)} holds, where $K^\pm(x^\pm(p))$ are defined in {\bf(H5)}.
By Theorem~\ref{uni}, (\ref{PPLF}) has a unique crossing periodic orbit, which is stable. Thus conclusion (i) is proved.

If $a_R/t_R<a_L/t_L$, applying the changes $(t, x, y)\rightarrow(-t, -x, y)$
and
\begin{eqnarray}
(t_L, d_L, a_L, t_R, d_R, a_R)\rightarrow(-t_R, d_R, -a_R, -t_L, d_L, -a_L)
\label{change}
\end{eqnarray}
to (\ref{PPLF}) we observe that
the form of (\ref{PPLF}) is invariant.
Thus conclusion (ii) is directly obtained from conclusion (i).

If $a_R/t_R=a_L/t_L$, then
$$\lim_{p\rightarrow0^+}\frac{g^+(x^+(p))}{f^+(x^+(p))}=-\frac{a_R}{t_R}=-\frac{a_L}{t_L}=\lim_{p\rightarrow0^+}\frac{g^-(x^-(p))}{f^-(x^-(p))}.$$
Define
\begin{eqnarray*}
\Lambda(p):=\frac{g^+(x^+(p))}{f^+(x^+(p))}-\frac{g^-(x^-(p))}{f^-(x^-(p))}
\end{eqnarray*}
for $p>0$.
For system (\ref{PPLF}), we get
\begin{eqnarray*}
\Lambda(p)
=\left(\frac{d_R}{t_R^2}p-\frac{a_R}{t_R}\right)-
\left(\frac{d_L}{t_L^2}p-\frac{a_L}{t_L}\right)=\left(\frac{d_R}{t_R^2}-\frac{d_L}{t_L^2}\right)p.
\end{eqnarray*}
When $d_R/t_R^2<d_L/t_L^2$, we have $\Lambda(p)<0$ for $p>0$, i.e., {\bf(H3)} holds. Moreover, $(0, 0)$ is the unique solution of equations (\ref{induceeq}).
Hence,
(\ref{PPLF}) has no crossing periodic orbits by Theorem~\ref{suff}.
When $d_R/t_R^2>d_L/t_L^2$, by the changes $(t, x, y)\rightarrow(-t, -x, y)$ and (\ref{change}),
the nonexistence of crossing periodic orbits is directly obtained
from the case of $d_R/t_R^2<d_L/t_L^2$.
Thus, $d_R/t_R^2=d_L/t_L^2$ is a necessary condition for the existence of crossing periodic orbits.
Then, if there exists a crossing periodic orbit, we have $\Lambda(p)\equiv0$ for all $p>0$, i.e.,
condition (\ref{noncon}) of Theorem~\ref{nonexist} is satisfied.
Therefore, there exists a periodic annulus including this crossing periodic orbit by Theorem~\ref{nonexist}.
Conclusion (iii) is proved.
\end{proof}

We remark that a similar result to Theorem~\ref{existandunique} is given in \cite[Theorem 4]{JEF}
for system (\ref{PPLF}) satisfying (\ref{atcon}) and $a_L>0>a_R$, which is not required in our Theorem~\ref{existandunique}.
So Theorem~\ref{existandunique} generalizes \cite[Theorem 4]{JEF} and this generalization is crucial
for us to prove Theorem~\ref{linear}. We will see this in the proof of Theorem~\ref{linear} later.

\begin{lm}
Assume that $b=0, d_L d_R\ne0$ in system {\rm(\ref{PPLF})}. Then there exist no crossing limit cycles if
$t_Lt_R\ge0$.
\label{linearnon}
\end{lm}

\begin{proof}
If $t_Lt_R\ge0$ and $t_L+t_R\ne0$, the result of no crossing limit cycles is obtained directly from \cite[Proposition 3.7]{ChL-EE}.
If $t_Lt_R\ge0$ and $t_L+t_R=0$, i.e., $t_L=t_R=0$, the equilibrium of the left (resp. right) system of (\ref{PPLF}) is either a center when $d_L>0$ (resp. $d_R>0$)
or a weak saddle (the sum of two eigenvalues is zero) when $d_L<0$ (resp. $d_R<0$).
By \cite[Theorems 2 and 4]{JDNT}, system (\ref{PPLF}) has no crossing limit cycles.
\end{proof}

In the end of this paper we give a proof of Theorem~\ref{linear}.

\begin{proof}[{\bf Proof of Theorem~\ref{linear}}]
As indicated in the second paragraph of this section, we can equivalently consider system (\ref{PPLF})
to investigate the existence, uniqueness and number of crossing limit cycles of discontinuous system
(\ref{generalPLF}). Furthermore, it is easy to verify that (\ref{generalPLF})
has no sliding sets if and only if (\ref{PPLF}) has no ones and that (\ref{generalPLF}) is nondegenerate if and only if
(\ref{PPLF}) is nondegenerate. Therefore, we only need to consider nondegenerate (\ref{PPLF}) without
sliding sets. By the nonexistence of sliding sets and nondegeneracy, (\ref{PPLF}) satisfies
$b=0$ and $d_R d_L\ne0$.

Totally there are $7$ cases
\begin{eqnarray*}
\begin{aligned}
&{\rm (C1)}~~a_L=a_R=0,~~~~~~ &&{\rm (C2)}~~a_L>0\ge a_R, \\
&{\rm (C3)}~~a_L<0\le a_R, && {\rm (C4)}~~a_L>0, a_R>0
\end{aligned}
\end{eqnarray*}
and
$$
{\rm (C5)}~~a_L=0, a_R<0, ~~~~~~{\rm (C6)}~~a_L=0, a_R>0, ~~~~~~{\rm (C7)}~~a_L<0, a_R<0.
$$
By the change
$$
(x, y, t, t_L, d_L, a_L, t_R, d_R, a_R)\rightarrow(-x, -y, t, t_R, d_R, -a_R, t_L, d_L, -a_L),
$$
(C5), (C6) and (C7) are transformed into (C2), (C3) and (C4), respectively.
Thus, we only need to consider (C1), $\cdot\cdot\cdot$, (C4).

Assume that (\ref{PPLF}) satisfies (C1). Then (\ref{PPLF}) is continuous,
where the definition of continuity is given below (\ref{LS}).
It is proved in \cite[Corollary 3]{ChL-EEFF} that continuous (\ref{generalPLF})
has at most one crossing limit cycle, so does (\ref{PPLF}).

Assume that (\ref{PPLF}) satisfies (C2). When $a_R=0$ and $t_R^2-4d_R\ge0$,
the equilibrium of the right system lies in the switching line $y$-axis
and it is neither focus nor center, implying that (\ref{PPLF}) cannot have crossing
limit cycles. When either $a_R=0, t_R^2-4d_R<0$ or $a_L>0>a_R$,
the origin $O$ is a $\Sigma$-monodromic singularity (see \cite{ChL-JJT}),
i.e., all orbits in a small neighborhood of $O$ turn around $O$.
Thus (\ref{PPLF}) also has at most one crossing limit cycle by \cite[Theorem 1.1]{ChL-JJT}.

Assume that (\ref{PPLF}) satisfies (C3). When $d_L<0$ or $d_R<0$, at least one of equilibria of the left and right systems is a saddle.
Moreover, this saddle lies in $x>0$ if it is of the left system and $x\le0$ if it is of the right one,
so that (\ref{PPLF}) has no crossing limit cycles.
When $d_R>0, d_L>0$ and $t_Lt_R\ge0$, (\ref{PPLF}) also has no crossing limit cycles by Lemma~\ref{linearnon}.
When $d_R>0, d_L>0$ and $t_L<0<t_R$, (\ref{PPLF}) satisfies condition (\ref{atcon}) in Theorem~\ref{existandunique}.
Thus (\ref{PPLF}) has at most one crossing limit cycle by Theorem~\ref{existandunique}.
When $d_R>0, d_L>0$ and $t_R<0<t_L$, by the change
$$(x, y, t, t_L, d_L, a_L, t_R, d_R, a_R)\rightarrow(x, -y, -t, -t_L, d_L, a_L, -t_R, d_R, a_R)$$
we obtain the uniqueness of crossing limit cycles from the case $d_R>0, d_L>0, t_L<0<t_R$.

Assume that (\ref{PPLF}) satisfies (C4).
Applying the change
\begin{eqnarray}
(t, x, y)\rightarrow
\left\{
\begin{aligned}
&(t/a_R, x/a_R, y)~~~~&&{\rm for}~x>0,\\
&(t/a_L, x/a_L, y)~~~~~&&{\rm for}~x\le 0
\end{aligned}
\right.
\label{achan}
\end{eqnarray}
to (\ref{PPLF}), we obtain
\begin{eqnarray}
\left(
\begin{array}{c}
\dot x\\
\dot y\end{array}\right)=
\left\{
\begin{aligned}
&\left(
\begin{array}{cc}
  t_R/a_R&-1\\
  d_R/a_R^2&0\\
\end{array}
\right)\left(
\begin{array}{c}
x\\
y\end{array}\right)-
\left(
\begin{array}{c}
  0\\
  1\\
\end{array}
\right)~~~~{\rm if}~x>0,\\
&\left(
\begin{array}{cc}
  t_L/a_L&-1\\
  d_L/a_L^2&0\\
\end{array}
\right)\left(
\begin{array}{c}
x\\
y\end{array}\right)-
\left(
\begin{array}{c}
  0\\
  1\\
\end{array}
\right)~~~~~{\rm if}~x<0.
\end{aligned}
\right.
\label{form}
\end{eqnarray}
Observing that (\ref{form}) is continuous, we know that (\ref{form}) has at most one crossing limit cycle by \cite[Corollary 3]{ChL-EEFF} again.
Finally, we conclude that (\ref{PPLF}) has at most one crossing limit cycle because (\ref{achan}) is a homeomorphism.

In conclusion, nondegenerate and discontinuous (\ref{generalPLF}) without sliding sets has at most one crossing limit cycle.
Next, we show the reachability of this number and the location of crossing limit cycles by considering the following nondegenerate and discontinuous system as an  example,
\begin{eqnarray}
\left(
\begin{array}{c}
\dot x\\
\dot y\end{array}\right)=
\left\{
\begin{aligned}
&\left(
\begin{array}{cc}
  2&-1\\
  2&0\\
\end{array}
\right)\left(
\begin{array}{c}
x\\
y\end{array}\right)-
\left(
\begin{array}{c}
  0\\
  2\\
\end{array}
\right)~~~~&&{\rm if}~x>0,\\
&\left(
\begin{array}{cc}
  -4&-1\\
  5&0\\
\end{array}
\right)\left(
\begin{array}{c}
x\\
y\end{array}\right)-
\left(
\begin{array}{c}
  0\\
  5\chi\\
\end{array}
\right)~~~~~&&{\rm if}~x<0,
\end{aligned}
\right.
\label{example}
\end{eqnarray}
where $\chi\in\{\epsilon, 0, 1\}$ and $-1\ll\epsilon<0$.

For the right system, its equilibrium $E_1:=(1, 2)$ is an unstable focus of (\ref{example}) and $O$ is a visible tangency point, i.e., the orbit of the right system passing through $O$ is tangent to $x=0$ at $O$ from the right side.
Then the forward orbit of the right system starting from $(0, y_0)$ with $y_0\le0$ evolves in the right half plane until it reaches again
$y$-axis at a point $(0, y_1)$ with $y_1>0$ after a finite time $t^+>0$. Hence we define the right Poincar\'e map $P_R(y_0):=y_1$.
As completed in \cite{ChL-EE, ChL-FEEE},
the parametric representation of $P_R$ is given by
$$y_0(t^+)=\frac{e^{-t^+}-\cos t^+}{\sin t^+}+1, ~~~~~y_1(t^+)=-\frac{e^{t^+}-\cos t^+}{\sin t^+}+1$$
for $t^+\in(\pi, \hat t^+]$ and
\begin{eqnarray}
P_R(0)=-2e^{\hat t^+}\sin \hat t^+>0,~~~~~~\lim_{y_0\rightarrow-\infty}P'_R(y_0)=-e^{\pi},
\label{eurej}
\end{eqnarray}
where $\hat t^+\in(\pi, 2\pi)$ satisfies $y_0(\hat t^+)=0$, i.e., the time for the orbit passing from $(0,0)$ to $(0, P_R(0))$ in the
right half plane.

For the left system, its equilibrium $E_2:=(\chi, -4\chi)$ is a stable focus. Moreover, $O$ is a boundary equilibrium if $\chi=0$ and an invisible (resp. a visible) tangency point if $\chi=1$ (resp. $\epsilon$), i.e., the orbit of the left system passing through $O$ is tangent to $x=0$ at $O$ from right (resp. left) side.
Then there exists $\hat z_0\ge0$ such that the forward orbit of the left system starting from $(0, z_0)$ with $z_0\ge\hat z_0$
evolves in the left half plane and reaches again $y$-axis at a point of form $(0, z_1)$ with $z_1\le0$ after a finite time $t^-\ge0$. Choosing $\hat z_0$ as the minimum one, we define the left Poincar\'e map $P_L(z_0):=z_1$.
From \cite{ChL-EE, ChL-FEEE} again, the expression of $P_L$ is given by
$$
z_0(t^-)=\chi\left(\frac{e^{2t^-}-\cos t^--2\sin t^-}{\sin t^-}\right),~~~~~
z_1(t^-)=-\chi\left(\frac{e^{-2t^-}-\cos t^-+2\sin t^-}{\sin t^-}\right)
$$
if $\chi=\epsilon, 1$ and
$$P_L(z_0)=-e^{-2\pi}z_0$$
if $\chi=0$, where $t^-\in(\pi, \hat t^-]$ (resp. $[0, \pi)$) for $\chi=\epsilon$ (resp. $1$) and $\hat t^-\in(\pi, 2\pi)$ satisfies $z_1(\hat t^-)=0$.
In addition, we have
\begin{eqnarray}
\hat z_0=\left\{
\begin{aligned}
&0&&{\rm if}~\chi\ne\epsilon,\\
&5\epsilon e^{2\hat t^-}\sin\hat t^- &&{\rm if}~\chi=\epsilon,
\end{aligned}
\right.
~~~~~~~~~~\lim_{z_0\rightarrow+\infty}P'_L(z_0)=-e^{-2\pi}.
\label{njnj}
\end{eqnarray}

Let $P(y_0):=P_L(P_R(y_0))$. From the first equality of (\ref{eurej}) and (\ref{njnj}), we have $\hat z_0<P_R(0)$ for any $\chi$ and $-1\ll\epsilon<0$. Thus
$P(y_0)$ is well defined for $y_0\le0$ and $P(0)<0$ due to $P_R(0)>0, P_L(\hat z_0)=z_1(\hat t^-)=0$. On the other hand, from the second equality of
(\ref{eurej}) and (\ref{njnj}) we have
$$\lim_{y_0\rightarrow-\infty}P'(y_0)=\lim_{y_0\rightarrow-\infty}P'_L(P_R(y_0))\cdot P'_R(y_0)=e^{-\pi}<1,$$
implying that $P(y_0)>y_0$ for $y_0$ closed to $-\infty$. Therefore, $P(y_0)$ has a fixed point in $y_0<0$ for any $\chi$,
i.e., system (\ref{example}) has
a crossing periodic orbit. Since (\ref{example}) has no sliding sets, then this crossing periodic orbit is a crossing limit cycle
by the first part of this proof, i.e., the reachability is proved.

Notice that for system (\ref{example}), $E_2$ is a boundary equilibrium if $\chi=0$ and a regular equilibrium if $\chi=\epsilon$, but it is not an equilibrium if $\chi=1$. Moreover, $O$ is a regular point if $\chi=1$, a boundary equilibrium if $\chi=0$ and a pseudo-equilibrium if $\chi=\epsilon$. Thus system (\ref{example}) exactly has one equilibrium $E_1$ if $\chi=1$, two equilibria $O(E_2)$ and $E_1$ if $\chi=0$, and three equilibria $E_1, O$ and $E_2$ if $\chi=\epsilon$, which eventually imply that the number of equilibria surrounded by this crossing limit cycle is exactly $1$ (resp. $2, 3$) if $\chi=1$ (resp. $0, \epsilon$). The proof is completed.
\end{proof}

\end{document}